\documentclass[11pt]{amsart}

\usepackage{a4}
\usepackage{amsmath, amssymb, amsthm, mathtools}
\usepackage{xfrac} 
\usepackage[centering]{geometry}
\usepackage{commath}
\usepackage{tikz}
\usepackage{amsmath}
\usetikzlibrary{automata,shapes,arrows,positioning,chains}
\usepackage{graphicx, color}
\usepackage{hyperref}
\usepackage{enumerate}
\usepackage[mathscr]{euscript}
\usepackage{mathrsfs}
\usepackage{pst-node}
\usepackage{tikz-cd}
\usetikzlibrary{shapes.geometric, arrows}
\usetikzlibrary{patterns}
\usepackage{contour}
\contourlength{1.5pt}
\tikzset{mynode/.style args={#1 | #2}{midway,%
node contents={\contour{white}{\textcolor{#1}{#2}}},%
font=\tiny,inner sep=0pt}}
\tikzstyle{line} = [draw, -latex']
\tikzstyle{start} = [rectangle, rounded corners, minimum width=1cm, minimum height=1cm,text centered, draw=black]
\tikzstyle{arrow} = [thick,->,>=stealth]
\usepackage{bbm}
\DeclareMathOperator{\var}{var}
\DeclareMathOperator{\Per}{Per}
\DeclareMathOperator{\Image}{Im}

\begin{document}
 \newtheorem{theorem}{Theorem}
 \newtheorem{lemma}[theorem]{Lemma}
 \newtheorem{corollary}[theorem]{Corollary}
 \newtheorem{problem}[theorem]{Problem}
 \newtheorem{conjecture}[theorem]{Conjecture}
 \newtheorem{definition}[theorem]{Definition}
 \newtheorem{prop}[theorem]{Proposition}
 \numberwithin{equation}{section}
 \numberwithin{theorem}{section}
 
 \newcommand{\ms}{\mathscr}
 \newcommand{\Z}{\mathbb{Z}}
 \newcommand{\vphi}{\varphi}
 \newcommand{\Ps}{\varphi}
 
 \date{\today}

\parskip=0.5ex

\title{$\overline{d}$-continuity for Countable State Shifts } 
\author{Jasmine Bhullar}
\thanks{The author was partially supported by NSF DMS-1554794 and DMS-2154378.}
\email{jbhullar@uh.edu}
\address{Department of Mathematics, University of Houston}
\allowdisplaybreaks

\topmargin 0in 
\headsep 0.4in
\textheight 8in 
\textwidth 5.95in

\maketitle

\begin{abstract}
For full shifts on finite alphabets, Coelho and Quas showed that the map that sends a Hölder continuous potential $\phi$ to its equilibrium state $\mu_\phi$ is $\overline{d}$-continuous. We extend this result to the setting of full shifts on countable (infinite) alphabets. As part of the proof, we show that the map that sends a strongly positive recurrent potential to its normalization is continuous for potentials on mixing countable state markov shifts.
\end{abstract}
\section{Introduction}\label{sec:intro}

The concept of $\overline{d}$-distance on the space of invariant measures on a shift space was introduced by Ornstein to study the isomorphism problem for Bernoulli shifts. Ornstein showed that entropy is a complete invariant for the class of Bernoulli shifts: two Bernoulli shifts are isomorphic if and only if they have the same entropy \cite{Orn73, Orn74}. A key tool in this proof was the $\overline{d}$-metric. Many ergodic properties are well-behaved with respect to this metric. The entropy function $\mu \mapsto h_{\mu}(\sigma)$ is $\overline{d}$-continuous. Moreover, the set of processes that are isomorphic to Bernoulli shifts is $\overline{d}$-closed. The more familiar topology on the space of Borel probability measures is the weak*-topology. The topology coming from the $\overline{d}$-metric refines the weak*-topology. 

If $S$ is a countable set, the full shift with alphabet $S$ is the space of all (one-sided or two-sided) sequences with symbols coming from $S$, together with the left shift map $\sigma$. For Markov shifts on finite alphabets, potentials with summable variations have a unique equilibrium state \cite{Wal75}. One can quickly prove that the map that sends a potential to its unique equilibrium state is continuous in the weak*-topology. For full shifts on a finite alphabet, Coelho and Quas proved that the map that sends a potential $\phi$ to its unique equilibrium state $\mu_\phi$ is continuous with respect to the $\overline{d}$-metric on the space of shift-invariant probability measures and a suitable metric on the space of potentials \cite{CQ98}. 

A natural question that arises is whether a similar result holds if we take the alphabet to be countably infinite. In this paper, all shift spaces that we consider are one-sided. We use thermodynamic formalism for countable-state Markov shifts (CMS) developed by Sarig \cite{Sar99, Sar01, BS03}. In this case, the classical notion of topological pressure must be replaced by the \emph{Gurevich pressure} $P_G(\phi)$, where the partition function is a sum taken only over those periodic orbits that begin in a ﬁxed 1-cylinder. 

A useful tool in studying equilibrium states is the RPF operator. For every $\phi$ with summable variations,  the Ruelle operator $L_{\phi}$ is defined as \[L_\phi(f)(x)= \sum_{\sigma y=x} e^{\phi(y)} f(y)\]
When $L_{\phi}(1)=1$, we call the potential `normalized'.  The Generalized Ruelle Perron Frobenius (RPF) theorem proved in \cite{Sar01} relates the modes of recurrence of potentials to the eigendata of the RPF operator corresponding to the eigenvalue $\lambda=e^{P_G(\phi)}$. Some important results relating to the thermodynamic formalism for CMS are summarized in sections \ref{setting} and \ref{TF for CMS}. When $\phi$ is `positive recurrent', the eigendata of the RPF operator and its dual give rise to an important invariant probability measure corresponding to the potential, called the RPF measure. For getting a result similar to the one given by Coelho and Quas for the countable-state case, we need to put an even stronger condition on our potential, namely `strong positive recurrence'.

\begin{theorem} \label{Potential and equilibrium states}
Let $\Sigma$ be the full shift on a countably infinite alphabet. Let $0<\theta<1$ and let $A_\theta$ be the set of $\theta-$weakly Hölder continuous strongly positive recurrent potentials with finite Gurevich presssure. Let $d_\theta$ be the metric associated to the Hölder norm on $A_\theta$. Let $\phi$ be a potential in $A_{\theta}$. Then for given $\epsilon>0$, there exists $\delta>0$ such that $\overline{d}(m_{\phi}, m_{\tau})< \epsilon$ whenever $d_{\theta}(\phi,\tau)< \delta$, where $\tau \in A_{\theta}$ and $m_{\phi}$ and $m_{\tau}$ are the corresponding RPF measures.
\end{theorem}

The notion of $g$-functions was introduced to ergodic theory in \cite{Kea72} and is a useful tool in thermodynamic formalism. Conformal measures for $g$-functions are called $g$-measures. Corresponding to a strongly positive recurrent potential $\phi$, we get a normalized potential, which is the logarithm of a $g$-function $g_{\phi}$. The RPF measure for $\phi$ is a $g$-measure for $g_{\phi}$. Theorem \ref{Potential and equilibrium states} follows as a corollary to Theorems \ref{g-functions and g-measures countable} and \ref{Potential and normalized potentials}.

The first result gives the continuous dependence of $g$-measures on their corresponding $g$-functions. 
\begin{theorem} \label{g-functions and g-measures countable}
Suppose $g$ is a continuous $g$-function on a full shift, $\Sigma$, on a countably infinite alphabet with the property that \begin{equation} \label{eqn:hypothesis}\sum_{n=r}^\infty \prod_{i=r}^n e^{-V_i(g)}= \infty, \end{equation} for some $r \geq 1$, where $V_n(g)$ is the n-th variation of $\log(g)$. Let $\nu_g$ be a $g$-measure. Then, for every $\epsilon > 0$, there exists $\delta>0$ such that if $h$ is a $g$-function on $\Sigma$ and $d(h,g):=\| \log(g)-\log(h)\|_{\infty} < \delta$, then for any $g$-measure $\nu_h$ corresponding to $h$, $\bar{d}(\nu_g,\nu_h) < \epsilon$. 
\end{theorem}

Condition (\ref{eqn:hypothesis}) is satisfied whenever $\sum_{i=1}^{\infty} V_i(g) < \infty$; in particular, Theorem \ref{g-functions and g-measures countable} applies whenever $\log(g)$ is weakly Hölder continuous. As a corollary of the proof of Theorem \ref{g-functions and g-measures countable}, we get that in this case the map $g \mapsto \nu_g$ is locally Lipschitz with respect to the metrics $d$ and $\overline{d}$.

\begin{theorem} \label{g-functions and g-measures Lipschitz}
Suppose $g$ is a continuous $g$-function on a full shift,  $\Sigma$, on a countably infinite alphabet such that $\sum_{i=2}^{\infty} V_i(g) < \infty$, where $V_n$ is the n-th variation of $\log(g)$. Let $\nu_g$ be a $g$-measure. Then there exists a constant $A_g$ such that if $h$ is any $g$-function on $\Sigma$ with $d(h,g)<\ln(2)$ and $g$-measure $\nu_h$, then \[\overline{d}(\nu_g,\nu_h) \leq A_g d(g,h).\]
\end{theorem}

We prove Theorems \ref{g-functions and g-measures countable} and \ref{g-functions and g-measures Lipschitz} in section \ref{g-functions}. The proof of Theorem  \ref{g-functions and g-measures countable} requires the construction of a measure on $\Sigma \times \Sigma$ that is a fixed point for a map using a limiting procedure. In Coelho and Quas' proof, the finiteness of the alphabet plays a critical role. The challenge is that since in our case $\Sigma$ is not compact, then the set of probability measures on $\Sigma \times \Sigma$ is not weak*-compact, and we cannot use the Schauder–Tychonoff Theorem and can no longer use estimates that depend on the number of states. We overcome this obstacle by constructing a sequence of measures that is `tight' and by defining the metric $d$ on the space of $g$-functions in terms of $\log g$ instead of $g$.

The second result gives the continuous dependence of the normalized potentials $\log(g_{\phi})$ on the corresponding potentials for countable state Markov shifts. Given a graph $G$ with set of vertices $S$ and a set of directed edges, the Topological Markov Shift (TMS) associated to $G$ is the collection of infinite admissible paths on this directed graph. If $S$ is a finite set, we call this a subshift of finite type (SFT) and if $S$ is a countable (infinite) set, the corresponding TMS is called a countable-state Markov shift (CMS). Countable-state Markov shifts are of interest as they arise as coding spaces for non-uniformly hyperbolic systems which cannot be coded by an SFT. 

\begin{theorem} \label{Potential and normalized potentials}
Let $\Sigma$ be a mixing CMS. Let $0<\theta<1$ and let $A_\theta$ be the set of $\theta-$weakly Hölder continuous strongly positive recurrent potentials with finite Gurevich pressure. Let $d_\theta$ be the metric associated to the Hölder norm on $A_\theta$ and let $\phi$ be a potential in $A_{\theta}$. Then for given $\epsilon > 0$, there exists a $\delta>0$ such that  \[\|\log(g_{\phi})-\log(g_{\tau})\|_{\infty} < \epsilon ~\text{whenever} ~\tau \in A_{\theta}~\text{and}~d_{\theta}(\phi,\tau)< \delta\]
\end{theorem}
The proof of Theorem \ref{Potential and normalized potentials} is based on the spectral gap property (SGP) of the Ruelle operator corresponding to a strongly positive recurrent potential. Section \ref{potentials} is dedicated to the proof of this result. 

Note that Theorem \ref{Potential and normalized potentials} holds for CMS, while Theorems \ref{Potential and equilibrium states} and \ref{g-functions and g-measures countable}  hold for full shifts only. The extension of Theorems \ref{Potential and equilibrium states} and \ref{g-functions and g-measures countable} to CMS will be explored in a forthcoming paper (\cite{BC}).

\subsection*{Acknowledgments} 
I would like to express my sincerest gratitude to my advisor, Vaughn Climenhaga for his support, guidance and encouragement. This project benefitted tremendously from the many discussions I had with him and his invaluable feedback. 

\section{Background}

\subsection{Setting} \label{setting}
Let $S$ be a countable set (alphabet). The \emph{one-sided full shift} $\Sigma_S$ with set of states $S$ is the set of sequences \[\Sigma_S = \lbrace (x_0x_1 \dots) \colon x_k \in S ~\text{for all}~k \in \mathbb{N} \cup \lbrace0\rbrace \rbrace\] together with the left shift map $\sigma: \Sigma_S \to \Sigma_S$ defined by $(\sigma(x))_k=x_{k+1}$ for all $k$. We equip $\Sigma_S$ with the product topology induced by the discrete topology on $S$. This topology is metrizable; one metric is given by $d(x,y)=e^{-t(x,y)}$ where $t(x,y) := \inf \lbrace k \colon x_k \neq y_k \rbrace$. The resulting topology is generated by \emph{cylinder sets} \[[a_0 \dots a_{n-1}] := \lbrace x \in X \colon x_i = a_i, i = 0, \dots n-1 \rbrace ~(n \in \mathbb{N}, a_i \in S).\] With respect to this topology, $\Sigma_S$ is a complete separable metric space, but it is not necessarily compact. A \emph{subshift} of $\Sigma_S$, also known as a \emph{shift space}, is a closed subset $\Sigma$ that is shift invariant ($\sigma(\Sigma)=\Sigma$). A \emph{finite word} over the alphabet $S$ is a finite sequence of symbols from $S$. The $language$ of a shift space $\Sigma$ is the set of finite words that appear in some sequence in $\Sigma$. Such words are called \emph{admissible}. 

Let $A=(t_{ij})_{S \times S}$ be a matrix of zeroes and ones, with columns and rows that are not all zeroes. The \emph{Topological Markov Shift (TMS)} with set of states $S$ and transition matrix $A$ is the shift space \[\Sigma := \lbrace (x_0,x_1, \dots) \in \Sigma_S \colon t_{x_i x_{i+1}}=1 ~\text{for all} ~i \rbrace.\] Corresponding to this, we can construct a directed graph with set of vertices $S$, and set of edges $\lbrace a \to b \colon t_{ab}=1 \rbrace$. We can then think of $\Sigma$ as the collection of infinite admissible paths on this directed graph. 
$\Sigma$ is compact if and only if the $S$ is finite. A \emph{subshift of finite type} is a compact Topological Markov Shifts. Non-compact TMS are also called \emph{Countable-state Markov Shifts (CMS)}.

We assume throughout that $\Sigma$ is \emph{topologically~mixing}, that is, for any two states $a$ and $b$, there is $N(a,b)$ such that for all $n \geq N(a,b)$, there is an admissible word of length $n+1$ which starts at $a$ and ends at $b$.

A function $\phi: \Sigma \to \mathbb{R}$ is called a \emph{potential}. The modulus of continuity of $\phi$ is captured by the decay of variations of $\phi$. The \emph{variations} for $\phi$ are \[\var_n(\phi)= \sup \lbrace |\phi(x)-\phi(y)| \colon x,y \in \Sigma, x_i = y_i, 0 \leq i \leq n-1\rbrace. \]
A potential $\phi$ is uniformly continuous if and only if $\var_n(\phi) \to 0$.  $\phi$ is called:
\begin{enumerate}
\item \textbf{$\theta$-weakly Hölder continuous} for some $0<\theta<1$ if there exists a constant $C_\phi > 0$ such that $\var_n(\phi) \leq C_{\phi} \theta^n$, for all $n \geq 2$. If in addition, $\var_1(\phi)$ is finite, $\phi$ is called \textbf{locally Hölder continuous}. Furthermore, if $\phi$ is bounded, then it is called \textbf{$\theta$-Hölder}. $\theta$ is called the \emph{Hölder exponent} of $\phi$.
\item A function with \textbf{summable variations} if $\sum_{n \geq 2} \var_n(\phi)< \infty$.
\end{enumerate} Note that the condition of being weakly Hölder continuous is stronger than summable variations. None of these two regularity conditions imply boundedness, and it is also possible that $\var_1(\phi)= \infty$. 

Let $0<\theta<1$ and let $\| \cdot \|_{\theta}$ denote the Hölder norm, which is given by \begin{equation} \label{Hölder norm} \| \phi \|_{\theta}:= \| \phi \|_{\infty}+ |\phi|_{\theta}, \end{equation} where $\| \cdot \|_{\infty}$ is the uniform norm and \begin{equation} \label{Hölder semi norm} |\phi |_{\theta}:= \sup \lbrace |\phi(x)- \phi(y)| / \theta^{t(x,y)} \colon x,y \in X, x \neq y\rbrace. \end{equation} Let $d_{\theta}$ be the metric coming from the norm $\| \cdot \|_{\theta}$. Let $A_{\theta}$ be the set of $\theta$-weakly Hölder continuous strongly positive recurrent potentials with finite Gurevich pressure (see Section \ref{TF for CMS}). Note that it is possible that $d_{\theta}(\phi, \psi)= \infty$ for potentials $\phi$ and $\psi$ in $A_{\theta}$.

\subsection{Thermodynamic formalism for CMS} \label{TF for CMS}
We use thermodynamic formalism for CMS developed in \cite{Sar99}, \cite{Sar01} and \cite{BS03}.  Let $\phi$ be a potential with summable variations. The \emph{Birkhoff sums} of $\phi$ are denoted by $S_n(\phi) := \sum_{k=0}^{n-1} \phi \circ \sigma^k$. For any state $a$, let 
\begin{align*}
\Per_n(\Sigma,a) &:= \lbrace x \in \Sigma \colon \sigma^n(x)=x, x_0=a \rbrace, ~\text{and}~\\
\Per^*_n(\Sigma,a) &:= \lbrace x \in \Per_n(\Sigma,a) \colon  x_i \neq a ~\text{for all}~ 1 \leq i \leq n-1 \rbrace. 
\end{align*}
We also write 
\begin{align*}
Z_n(\phi, a) &= \sum_{x \in \Per_n(\Sigma,a)} e^{S_n\phi(x)},~\text{and}\\
Z_n^*(\phi,a) &= \sum_{x \in \Per^*_n(\Sigma,a)} e^{S_n\phi(x)}.
\end{align*}
The limit $P_G(\phi) := \lim_{n \to \infty} \frac{1}{n} \ln Z_n(\phi,a)$ exists for all states $a$ and is independent of the choice of state $a$. $P_G(\phi)$ is called the \emph{Gurevich pressure} of $\phi$. When $\phi=0$, we obtain the \emph{Gurevich entropy} $h_G(\Sigma)= P_G(0)$. 

It was proved in \cite[Theorem 3]{Sar99} that if $\|L_{\phi}(\mathbbm{1}) \|_{\infty} < \infty$, then the Gurevich pressure satisfies a variational principle: \[P_G(\phi)= \sup \left\lbrace h_\mu(\sigma)+ \int \phi \,d\mu \right\rbrace,\] where the supremum ranges over all invariant probability measures such that the sum is not of the form $\infty - \infty$. A measure $\mu_{\phi}$ which attains this supremum is called an \emph{equilibrium state} for $\phi$. 

\subsubsection{Modes of Recurrence, Ruelle-Perron-Frobenius (RPF) operator and the Generalised RPF Theorem} \label{grpf}

  Let $P_G(\phi) < \infty$ and set $\lambda :=e^{P_G(\phi)}$. Following Sarig, we say that $\phi$ is \emph{recurrent} if for some state $a, \sum_{n \geq 1} \lambda^{-n} Z_n(\phi,a)$ diverges and \emph{transient} if it converges. Further, if $\phi$ is recurrent and $\sum_{n \geq 1}n \lambda^{-n} Z^*_n(\phi,a)$ converges, $\phi$ is called \emph{positive recurrent}.  $\phi$ is said to be \emph{strongly positive recurrent (SPR)} \footnote{Sarig's original definition of SPR potentials is given below. \cite[Section 8.5]{Cli18} shows that \ref{SPR condition} is equivalent.} if 
  \begin{equation} \label{SPR condition} \limsup_{n \to \infty} \frac{1}{n} \ln Z_n^*(\phi, a) < P_G(\phi). \end{equation}
These definitions can be shown to be independent of the choice of the state $a$.

For every $\phi$ with summable variations, we associate the RPF operator $L_{\phi}$ which is defined as \[L_\phi(f)(x)= \sum_{\sigma y=x} e^{\phi(y)} f(y).\]

The Generalised RPF theorem proved in \cite{Sar01} relates the modes of recurrence to the eigendata of the RPF operator. We have that $\phi$ is recurrent if and only if there exists a $\lambda>0$, a positive continuous function $h_{\phi}$ and a conservative measure $\nu_{\phi}$ which is finite and positive on cylinders such that \[L_{\phi}h_{\phi}=\lambda h_{\phi}~\text{and}~~~~~L_{\phi}^*\nu_{\phi}=\lambda \nu_{\phi}.\]
Here, $\lambda=e^{P_G(\phi)}$. $\phi$ is positive recurrent if and only if $\int h_{\phi} \, d\nu_{\phi} < \infty$. The measure $dm_{\phi}=h_{\phi} d\nu_{\phi}$ is called the RPF measure for $\phi$. If $\phi$ is positive recurrent and its RPF measure has finite entropy, then it is an equilibrium state for $\phi$. For subshifts of finite type, this is always the case. Furthermore, if the equilibrium state exists, then it is unique (\cite{BS03}). 

We now give Sarig's original definition of the strong positive recurrence condition in terms of a \emph{discriminant}.  Let $\Sigma$ be a TMS and $a \in S$. We define an induced system on $[a]$ as $\sigma_a \colon \Sigma_a \to \Sigma_a$ where $\Sigma_a= \lbrace x \in \Sigma \colon x_0=a, x_i = a ~\text{infinitely often} \rbrace$ and $\sigma_a= \sigma^{\tau_a(x)}(x)$, where $\tau_a(x)= \min \lbrace n \geq 1 \colon x_n=a \rbrace$.

This system can be given the structure of a TMS as follows. Let $\overline{S}=\lbrace [a, \xi_1, \cdots, \xi_{n-1}, a]: n \geq 1, \xi_i \neq a \rbrace \setminus \emptyset$. Let $\overline{\Sigma} = \overline{S}^{\mathbb{N} \cup \lbrace 0 \rbrace}$ and $\overline{\sigma}$ be the left shift map on $\overline{\Sigma}$. Then $(\overline{\Sigma}, \overline{\sigma})$ is topologically conjugate to $(\Sigma_a, \sigma_a)$ and the conjugacy $\pi: \overline{\Sigma} \to \Sigma_a$ is given by: 
\[\pi([a,\xi^0,a],[a,\xi^1,a],\dots) := (a,\xi^0,a,\xi^1, a \dots).\]

Corresponding to every potential function $\phi$, we get an induced potential $\overline{\phi}: \overline{\Sigma} \to \mathbb{R}$, which is defined as \[\overline{\phi} := \left(\sum_{i=0}^{\tau_a-1} \phi \circ \sigma^i \right) \circ \pi. \]
If $\phi$ is weakly Hölder continuous, then $\overline{\phi}$ is locally Hölder continuous. 
The \emph{$a$-discriminant} of $\phi$ is defined as: 
\[\Delta_a[\phi] := \sup \lbrace P_G(\overline{\phi+p}) \colon p \in \mathbb{R} ~\text{s.t.}~ P_G(\overline{\phi+p})< \infty \rbrace.\]
$\phi$ is said to be \emph{strongly positive recurrent} if there exists a state $a$ such that $\Delta_a[\phi]>0$. If $\phi$ is strongly positive recurrent, then $P_G(\phi)=0 \iff P_G(\overline{\phi})= 0$.

\subsection{$\overline{d}$-metric}
 Ornstein introduced the concept of $\overline{d}$-distance on the space of invariant measures on a shift space to study the isomorphism problem for Bernoulli shifts. He also proved that the set of processes which are measure theoretically isomorphic to Bernoulli shifts is $\overline{d}$-closed (see \cite{Orn73}, \cite{Orn74}). 

Let $(\Sigma,\sigma)$ be a CMS. We denote the space of Borel probability measures on $\Sigma$ by $\ms{M}(\Sigma)$ and let $\ms{M}^{\sigma} (\Sigma)$ denote the space of shift invariant Borel probability measures on $\Sigma$. For $\nu_1$ and $\nu_2$ in $\ms{M}^{\sigma} (\Sigma)$ , a \emph{joining} of $(\Sigma,\nu_1)$ and $(\Sigma,\nu_2)$ is a shift invariant measure on $\Sigma \times \Sigma$ with the property that $\mu(A \times \Sigma)= \nu_1(A)$ and $\mu(\Sigma \times B)=\nu_2(B)$ for any measurable subsets $A$ and $B$ of $\Sigma$. Equivalently, a joining of $(\Sigma,\nu_1)$ and $(\Sigma,\nu_2)$  is a shift invariant measure on $\Sigma \times \Sigma$ such that $\pi_1^*\mu=\nu_1$ and  $\pi_2^*\mu=\nu_2$, where $\pi_1, \pi_2$ are projections onto the two coordinates. We now define $\overline{d}$-metric in terms of joinings. The following definition is shown to be equivalent to Ornstein's original definition in \cite{Rud90}.

\begin{definition}
Let $J(\nu_1,\nu_2)$ be the set of joinings of $(\Sigma,\nu_1)$ and $(\Sigma,\nu_2)$. Let $\delta(x,y)$ be 1 if $x_0 \neq y_0$ and $0$ otherwise. Then  \begin{align*}
\bar{d}(\nu_1,\nu_2) &= \inf_{\mu \in J(\nu_1,\nu_2)} \int \delta(x,y) \, d\mu(x,y) \\
&= \inf_{\mu \in J(\nu_1,\nu_2)} \mu(\lbrace (x,y) \in \Sigma \times \Sigma \colon x_0 \neq y_0 \rbrace) 
\end{align*}
\end{definition}

The more familiar topology on the space of Borel probability measures on a metric space $Y$ is the weak* topology. A sequence $\mu_n$ in $\ms{M}(Y)$ is said to converge weak* to $\mu$ if $\int f \,d\mu_n \to \int f \,d\mu$ for all bounded continuous function $f$ on $Y$. When $Y$ is compact, $\ms{M}(Y)$ is compact in the weak* topology. When $Y$ is not compact, it is important to first understand under what conditions does a sequence in $\ms{M}(Y)$ have a convergent subsequence. 

A subset of $\ms{M}(Y)$ is said to be \textbf{tight} if for every $\epsilon > 0$, there exists a compact subset $K_\epsilon$, such that $\mu(K_{\epsilon}^c)< \epsilon$ for all measures $\mu$ in the subset. If $Y$ is a complete separable metric space and $\mu_n$ is a tight sequence in $\ms{M}(Y)$, then the Helly-Prokhorov theorem gives that it has a subsequence that converges  weak* to some $\mu$ in $\ms{M}(Y)$. Hence, every tight subset of $\ms{M}(Y)$ is precompact. 

\begin{lemma} \label{compactness}
For $\nu_1$ and $\nu_2$ in $\ms{M}^{\sigma} (\Sigma)$, $J(\nu_1, \nu_2)$ is weak*-compact and hence the infimum in the above definition of $\overline{d}$-metric is actually a minimum.
\end{lemma}

\begin{proof}

Since every finite Borel measure on a complete separable metric space is tight, then for given $\epsilon > 0$ we get compact sets $A$ and $B$ in $\Sigma$ such that $\nu_1(\Sigma \setminus  A) < \epsilon$ and $\nu_2(\Sigma \setminus B)<\epsilon$. 

Let $K=A \times B \subset \Sigma \times \Sigma$ and take $m$ in $J(\nu_1, \nu_2)$. Since \[(\Sigma \times \Sigma) \setminus (A \times B) \subset (\Sigma \setminus A) \times \Sigma) \cup (\Sigma \times (\Sigma \setminus B) \] 
We conclude that
\[m((\Sigma \times \Sigma) \setminus (A \times B)) \leq \nu_1(\Sigma \setminus A) + \nu_2(\Sigma \setminus B) < 2\epsilon \]

Therefore, $J(\nu_1, \nu_2)$ is tight, hence it is pre-compact in the weak*-topology. Also, $J(\nu_1, \nu_2)$ is closed. Indeed, let $\mu_n$ be a sequence in $J(\nu_1,\nu_2)$ converging to $\nu$ and consider for $f \in C_b(X)$,
\[\int f \circ \pi_1 \,d\nu = \lim_{n \to \infty} \int f \circ \pi_1 \,d\mu_n= \lim_{n \to \infty} \int f \, d\nu_1= \int f d\nu_1\]

Hence, $(\pi_1)_*\nu=\nu_1$ and similarly $(\pi_2)_*\nu=\nu_2$. So, $\nu$ is in $J(\nu_1, \nu_2)$.
Hence, $J(\nu_1, \nu_2)$ is compact. 
\end{proof}

\section{$g$-function and $\overline{d}$-continuity} \label{g-functions}

\subsection{Definitions and basic facts}

The notion of $g$-functions and $g$-measures was introduced to ergodic theory in \cite{Kea72} and is a useful tool in thermodynamic formalism. In \cite{Wal75}, Walters gave a proof of Ruelle's operator theorem using $g$-measures. 

Let $\Sigma$ be a full shift with a countable set of states $S$ and let $\sigma$ be the left shift map on $\Sigma$. A continuous function $g \colon \Sigma \to (0,1]$ is called a \emph{$g$-function} if \[\sum_{y \in \sigma^{-1}(x)} g(y)=1 , ~ \text{for all} ~ x \in \Sigma.\]

Conformal measures for $g$-functions are called $g$-measures, that is $\mu$ is a $g$-measure if and only if $L_g^*(\mu)=\mu$. For a given $g$-function, a \emph{$g$-measure} is a $\sigma$-invariant measure $\mu$ in the space of Borel probability measures of $\Sigma$ such that \[g(x)= \lim_{|A| \to 0, x \in A} \frac{\mu(A)}{\mu(\sigma(A))}, \mu \text{-a.e}\]
where $|A|$ denotes the diameter of $A$.

We denote by $[x]^n= \lbrace y: y_k=x_k$, for all $k \leq n \rbrace$, the cylinder of those points which agree with $x$ for the first $n+1$ places. The measure $\mu$ is a $g$-measure if \[\displaystyle\lim_{n \to \infty} \frac{\mu([ix]^{n+1})}{\mu([x]^n)}= g(ix),~ \mu \text{-a.e}.\]  
 
We now discuss a probabilistic interpretation of $g$-functions and $g$-measures. We can view $g(ix)$ as a Markov transition probability, giving the probability of moving from $x$ to $ix$ and a $g$-measure is an invariant measure for this Markov process (see \cite{Qua96}). More precisely, let us consider a sequence of random variables $(X_n)_{n \in \mathbb{Z}}$ taking values in the set of states $S$. $X_n$ are maps from some probability space $\Omega$ to $S$. Suppose the sequence satisfies  \[\mathbb{P}(X_n=i| X_{n-1}=x_1,X_{n-2}=x_2, \dots)= g(ix).\] We fix a natural number $n$ and define $\rho_n \colon \Omega \to \Sigma$ by $(\rho_n(\omega))_i = X_{n-i}(\omega)$. The probability distribution on $(X_m)_{m \leq n}$ pushes forward to a probability measure on $\Sigma$ where $\rho_n^*(\mathbb{P}(A))= \mathbb{P}(\rho_n^{-1}(A)))$. $g$-measures correspond to stationary distributions for the random variables. 

For SFTs, continuous functions have at least one $g$-measure associated to them but uniqueness is not guaranteed. In the cases where $g$ is Hölder continuous or the variations of $g$ are summable, the $g$-measure is known to be unique (proofs can be found in \cite{Kea72}, \cite{Wal75}).

Let $\phi$ be a recurrent potential, and let  $\lambda=e^{P_G(\phi)}$. Let $h_{\phi}$ be given by the Generalized Ruelle Perron Frobenius theorem. Then $g_{\phi}:= \frac{e^{\phi}h_{\phi}}{\lambda h_{\phi} \circ \sigma}$ is a $g$-function. The RPF measure for $\phi$ is a $g$-measure for $g_{\phi}$. 

The following lemmas give some results about the RPF operator associated to continuous $g$-functions. 
 
 \begin{lemma} \label{transfer}
Let $g$ be a continuous $g$-function on $\Sigma$ and  $L_{\log g}: C(\Sigma) \to C(\Sigma)$ be the RPF operator corresponding to $\log g$. Then $L_{\log g}$ is continuous with respect to the uniform norm on $C_b(\Sigma)$.
\end{lemma}

\begin{proof}
$\| L_{\log g}(f) \|_\text{sup}= \sup_{x \in \Sigma} \lvert L_{\log g}(f)(x) \rvert = \sup_{x \in \Sigma} \lvert \sum_{y \in \sigma^{-1}(x)} g(y) f(y) \rvert $

But\[\lvert f(y) \rvert \leq \|f\|_{\text{sup}},~\text{for all}~ y \in \Sigma, ~\text{and}\]
\[\sum_{y \in \sigma^{-1}(x)} g(y)=1,~\text{for all}~ x \in \Sigma\]
Hence, \[\| L_{\log g}(f) \|_\text{sup} \leq \|f\|_{\text{sup}} \qedhere \]
\end{proof} 
The following result about dual operators on Banach spaces shows that $L_{\log g}^*$ is also continuous. 
\begin{lemma} \label{dual}
Let $T:V \to V$ be a bounded linear operator on a Banach space $V$. Let $T^*: V^* \to V^*$ be the dual of $T$, defined by 
\[T^*(f)= f \circ T\]
 Then $T^*$ is continuous with respect to the weak*-topology on $V^*$.
\end{lemma}

\begin{proof}
Let $(f_{\alpha})_{\alpha}$ be a net in $V^*$ with weak*-limit $f$, then 
\begin{align*}
&f_{\alpha}(Tx) \to f(Tx),  \forall x ~\text{in}~ V 
\implies (f_{\alpha} \circ T)(x) \to (f \circ T)(x),  \forall x ~\text{in}~ V \\
&\implies f_{\alpha} \circ T \xrightarrow{weak*}  f \circ T 
\implies T^*(f_{\alpha}) \xrightarrow{weak*} T^*(f) \qedhere
\end{align*}
\end{proof}

\subsection{Proof of Theorem \ref{g-functions and g-measures countable}}

Let $\Sigma$ be the full shift with countable alphabet $S$. 
We define the following on the space of continuous $g$-functions. 
\begin{align*}
        V_n(g) &= \sup_{[x]^{n-1}=[y]^{n-1}}  |\log(g(x))-\log(g(y))|, \\
       d(h,g) &= \sup_x |\log(g(x))-\log(h(x))|= \|\log(g)-\log(h) \|_{\infty}
 \end{align*}
Note that when $d(h,g)< \infty$, there exists $c$ in $(0,1]$ such that \begin{equation} \label{full shift estimate} g(ix) \geq c h(iy) ~\text{and}~ h(iy) \geq c g(ix) ~\text{for all}~ i,x,y. \end{equation}

In the remainder of this section, we give a proof of Theorem \ref{g-functions and g-measures countable}, which extends Theorem 1 in \cite{CQ98} to the setting of full shifts on countable alphabets. 

\begin{proof}[Proof of Theorem \ref{g-functions and g-measures countable}]

We follow the proof of Coelho and Quas \cite{CQ98}. However, various modifications are needed since the alphabet is infinite and the space is no longer compact. To avoid estimates that depend on the number of states, we use a metric coming from the difference in the logarithms of the $g$-functions to get multiplicative control. The proof also requires the construction of a joining using a limiting procedure. Since the space of measures we work with is no longer compact, we construct a sequence of measures that is `tight' to get the joining as a limit point. 

We divide the proof into three parts. In Section \ref{tp}, we first define transition probabilities whose marginals are given by the g-functions. In Section \ref{joining}, we use these transition probabilities to get a joining of the corresponding g-measures. Finally, in Section \ref{final}, we use this joining to get an estimate on the $\overline{d}$-distance between the two g-measures to complete the proof.  The arguments in sections \ref{tp} and \ref{joining} are valid for a general CMS. The arguments in section \ref{final} depend on the equation (\ref{full shift estimate}), which is only valid for the full shift. 

\subsubsection{Transition Probabilities} \label{tp}
Let $F:S \times \Sigma \times \Sigma \to (0,1)$ be defined by \[F(i,x,y)=\min(g(ix),h(iy))\]. Let $\Delta: \Sigma \times \Sigma \to [0,1)$ be given by \[\Delta(x,y)=1-\sum_i F(i,x,y).\] 
Let $d(g,h)<\infty$ and set \[c(g,h)=\inf_{i,x,y} \left( \min\left( \frac{g(ix)}{h(iy)}, \frac{h(iy)}{g(ix)} \right) \right).\]	
Then, for all $i,x,y$
\[	F(i,x,y) \geq c(g,h) g(ix). \]
So we get that \[   \Delta(x,y)=1-\sum_i F(i,x,y) \geq c(g,h) \leq 1- e^{-(V_1(g)+d(h,g))}. \]
Let $\alpha=1-e^{-(V_1(g)+d(h,g))}$.	
For $[x]^{k}=[y]^{k}$, 
\begin{align} \label{upper}
	\Delta(x,y) &\leq 1-e^{-(V_{k+1}(g)+d(h,g))}.
	\end{align}	

We define $G:\Sigma \times \Sigma \to (0,1]$ as follows: 
	\begin{align}
	G(ix,iy) &= F(i,x,y), ~\text {and} \label{transition1} \\
	G(ix,jy) &= \frac{(g(ix)-F(i,x,y))^+(h(jy)-F(j,x,y))^+}{\Delta(x,y)} \label{transition2}
	\end{align}
	for $i \neq j$ and $a^+=\text{max}(a,0)$.
	
\begin{lemma} \label{marginals}
For $G$ defined by equations (\ref{transition1}) and (\ref{transition2}), we have that 
\begin{align*}
\sum_i G(ix,jy) = h(jy) ~\text{ for all}~j \in S ~\text{and}~ x,y \in \Sigma ~\text{and} \\
\sum_j G(ix,jy) = g(iy) ~\text{ for all}~i \in S, ~\text{and}~ x,y \in \Sigma.
\end{align*} 
\end{lemma}

	\begin{proof}
	We fix $x,y$ and partition the set $S$ as follows $S=S_g \cup S_h \cup S_0$ where $S_0=\lbrace i \in S: g(ix)=h(iy) \rbrace, S_g= \lbrace i \in S: g(ix)>h(iy)\rbrace, S_h=\lbrace i \in S: g(ix)<h(iy) \rbrace$. 
	\begin{align*}
	\sum_i F(i,x,y) &= \sum_i \min(g(ix),h(iy)) \\
	&= \sum_{i \in S_h \cup S_0} g(ix) + \sum_{i \in S_g} h(ix) \\
	&= \sum_{i \in S_h \cup S_0} g(ix) + 1 - \sum_{i \in S_h \cup S_0} h(ix) \\	
	&= \sum_{i \in S_h \cup S_0} [g(ix)-h(iy)] +1\\
	\end{align*}
	So, we get that
	\begin{align*}
\Delta(x,y) &= - \sum_{i \in S_h \cup S_0} [g(ix)-h(iy)]\\
	\Delta(x,y) &=  \sum_{i \in S_h} [h(iy)-g(ix)]\\
	&= \sum_{i} [h(iy)-F(i,x,y)]^+ 
	\end{align*}
Finally, we get that 	\begin{align*}
	\sum_j G(ix,jy) &= \sum_{j \neq i} \frac{(g(ix)-F(i,x,y))^+(h(jy)-F(j,x,y))^+}{\Delta(x,y)}\\
	&+F(i,x,y) \\
	&= (g(ix)-F(i,x,y))^+ + F(i,x,y) \\
	&= g(ix) - F(i,x,y) +F(i,x,y) \\
	&=g(ix)
	\end{align*}
	Similarly we have\[\sum_i G(ix,jy)= h(jy)  \qedhere \] 
	\end{proof}
	
		\begin{lemma} \label{fixed point}
	Let $\pi_1: \Sigma \times \Sigma \to \Sigma$ be the projection onto the first coordinate, that is, $\pi_1(x,y)=x$. Then the following diagram commutes:
	\[ \begin{tikzcd}
\ms{M}(\Sigma \times \Sigma) \arrow{r}{L_{\log G}^*} \arrow[swap]{d}{(\pi_1)_*} & \ms{M}(\Sigma \times \Sigma) \arrow{d}{(\pi_1)_*} \\%
\ms{M}(\Sigma) \arrow{r}{L_{\log g}^*}& \ms{M}(\Sigma)
\end{tikzcd}
\]

Similarly, let $\pi_2: \Sigma \times \Sigma \to \Sigma$ be the projection onto the second coordinate, that is, $\pi_2(x,y)=y$. Then for all $m \in \ms{M}(\Sigma \times \Sigma)$, \[(\pi_2)_*(L_{\log G}^*(m))= L_{\log g}^*((\pi_2)_*(m))\]

In particular, if $m$ in $\ms{M}(\Sigma \times \Sigma)$ satisfies $(\pi_1)_{*}(m)=\nu_g$ and $(\pi_2)_{*}(m)=\nu_h$, then $(\pi_1)_{*}(L_{\log G}^*m)=\nu_g$ and $(\pi_2)_{*}(L_{\log G}^*m)=\nu_h$.
		\end{lemma}
		\begin{proof}
		Consider for $m \in \ms{M}(\Sigma \times \Sigma)$ and $f \in C_b(\Sigma)$, 
		\begin{align*}
		\int f \,d((\pi_1)_*(L_{\log G}^*(m))) &= \int L_{\log G} (f \circ \pi_1) \,dm. 
		\end{align*} 
		For $x,y \in \Sigma$, 
		\begin{align*}
		L_{\log G} (f \circ \pi_1)(x,y) &= \sum_{i,j} G(ix,jy) ((f \circ \pi_1)(ix,jy)) \\
		&= \sum_{i,j} G(ix,jy) f(ix) \\
		&= \sum_{i} g(ix) f(ix)~ (\text{by Lemma}~ \ref{marginals})\\
		&= L_{\log g}(f)(x) \\
		&= (L_{\log g}(f) \circ \pi_1)(x,y).
		\end{align*}
		Hence, we get that
		\begin{align*}
		\int f \,d((\pi_1)_*(L_{\log G}^*(m))) &= \int L_{\log G} (f \circ \pi_1) \,dm \\
		&= \int (L_{\log g}(f) \circ \pi_1) \,dm \\
		&= \int f \, d(L_{\log g}^* (\pi_1)_* m).
		\end{align*}
	Hence, $(\pi_1)_*(L_{\log G}^*(m))= L_{\log g}^*((\pi_1)_*(m))$. 
	
	Similarly, $(\pi_2)_*(L_{\log G}^*(m))= L_{\log g}^*((\pi_2)_*(m))$. Finally if $m$ in $\ms{M}(\Sigma \times \Sigma)$ satisfies $(\pi_1)_{*}(m)=\nu_g$, then 
	\[(\pi_1)_{*}(L_{\log G}^*m)=L_{\log g}^*((\pi_1)_*m)=L_{\log g}^*(\nu_g)=\nu_g.\]
	Similarly, we get $(\pi_2)_{*}(L_{\log G}^*m)=\nu_h$. \qedhere
	\end{proof} 
	\subsubsection{Joining} \label{joining}
	
	We now produce a joining of $\nu_g$ and $\nu_h$. For notational convenience, let $T:=L_G^*$. Then $T$ is an affine map of $\ms{M}(\Sigma \times \Sigma)$. $T$ is weak*-continuous (Lemmas \ref{dual} and \ref{transfer}). Our problem now translates to finding a fixed point for $T$. We do this by defining a tight sequence of measures. 

Let \[P_{\nu_g, \nu_h} = \lbrace m \in \ms{M}(\Sigma \times \Sigma) \colon m ~\text{has marginals}~ \nu_g ~\text{and}~ \nu_h \rbrace. \]
Then, take $m_0= \nu_g \times \nu_h \in P_{\nu_g, \nu_h}$.
Define a sequence as \[m_n= \frac{1}{n} \sum_{k=0}^{n-1} T^k(m_0).\]
Then by Lemma \ref{fixed point}, for all $n$, $m_n$ is in $P_{\nu_g, \nu_h}$.

By Lemma \ref{compactness}, there is a convergent subsequence, say $(m_{n_j})_j$. Let $\mu$ be the weak*-limit of this subsequence. We claim that $T\mu=\mu$.  Then since $\mu$ is a fixed point for the dual of the RPF operator, we get that $\mu$ is invariant for the map $\sigma \times \sigma$.

\begin{lemma} \label{fixed point}
 Let $\mu$ be the weak*-limit of a subsequence of the sequence\[m_n= \frac{1}{n} \sum_{k=0}^{n-1} T^k(m_0).\] Then $\mu$ is a fixed point for $T$. 
\end{lemma}

\begin{proof}
Let $K \subset \Sigma \times \Sigma$ be compact. Then the weak*-topology on $M(K)$, the space of measures on $K$, is metrizable. One metric that induces the weak*-topology on $M(K)$ is given by the following; Let $(f_n)_n \subset C(K)$ be a dense subset with $\| f_n \|_{\infty} =1$ for all $n$. We can use this to define a metric $d_K$ by \[d_K(\mu,\nu)= \sum_{n=1}^{\infty}\frac{|\int_K f_n \, d\mu- \int_K f_n \, d\nu |}{2^n}.\]
This defines a pseudo-metric on $\ms{M}(\Sigma \times \Sigma)$ and
 
  \[d_K(T\mu,\mu) \leq d_K(T\mu,Tm_{n_j})+d_K(Tm_{n_j},m_{n_j})+d_K(m_{n_j},\mu). \]
  
Any continuous function $f$ on $K$ can be extended to a continuous function on $\Sigma \times \Sigma$, as $K$ is compact. Also, $\mu$ is the limit of the subsequence $(m_{n_j})$ in the weak*-topology on $\ms{M}(\Sigma \times \Sigma)$, so we get that $d_K(m_{n_j},\mu) \to 0$. Since $T$ is continuous, $d_K(T\mu,Tm_{n_j}) \to 0$. 

For the middle term, observe the following:
\begin{align*}
m_{n_j} &= \frac{1}{n_j} \sum_{k=0}^{n_j-1} T^k(m_0), \\
Tm_{n_j} &= \frac{1}{n_j} \sum_{k=1}^{n_j} T^k(m_0), \\
d_K(Tm_{n_j},m_{n_j}) &= d_K \left(\frac{1}{n_j} \sum_{k=0}^{n_j-1} T^k(m_0), \frac{1}{n_j} \sum_{i=1}^{n_j} T^i(m_0) \right) \\
&=  \sum_{n=1}^{\infty}\frac{\left| \int_K f_n \, d\left( \frac{1}{n_j} \sum_{k=0}^{n_j-1} T^k(m_0) \right)- \int_K f_n \, d\left( \frac{1}{n_j} \sum_{i=1}^{n_j} T^i(m_0) \right)  \right|}{2^n} \\
&= \sum_{n=1}^{\infty}\frac{\frac{1}{n_j} \left| \int_K f_n \, d\left( m_0 \right)- \int_K f_n \, d\left(  T^{n_j}(m_0) \right) \right|}{2^n} \\
&\leq  \sum_{n=1}^{\infty}  \frac{1}{2^n n_j}  \left( \left| \int_K f_n \, d\left( m_0 \right) \right| + \left| \int_K f_n \, d\left(  T^{n_j}(m_0) \right) \right| \right) \\
&=  \sum_{n=1}^{\infty}  \frac{1}{2^n n_j}  \left( m_0(K) +  T^{n_j}(m_0)(K) \right) \\
&=  \frac{1}{n_j}  \left( m_0(K) +  T^{n_j}(m_0)(K) \right) \leq \frac{2}{n_j} \to 0.
\end{align*}

Hence, $d_K(T\mu,\mu)=0$, therefore $T\mu|_K=\mu|_K$. Since the compact subset K was arbitrarily chosen, $T\mu=\mu$. \qedhere
\end{proof} 

\subsubsection{Estimating $\overline{d}(\nu_g,\nu_h)$} \label{final}
We want to get an upper bound on $\bar{d}(\nu_g,\nu_h)$. We do this by estimating $\int \delta(x,y) \, d\mu$, where 
\begin{equation*} 
\delta(x,y)=
\begin{cases} 
      1 & x_0 \neq y_0 \\
      0 & \text{otherwise}   \end{cases}
\end{equation*}
Let $\pi(x,y)= \inf \lbrace n \geq 0: x_n \neq y_n \rbrace$, then $\delta(x,y)=1$ if and only if $\pi(x,y)=0$. 

Let $(X_n)_{n \in \mathbb{Z}}$ and $(Y_n)_{n \in \mathbb{Z}}$ be stochastic processes taking values in the set of states $S$. Let $X_{n-}$ denote the sequence $(X_{n-1},X_{n-2} \dots) \in \Sigma$ and $Y_{n-}$ be similarly defined. Let the transition probabilities for the pair $(X_n,Y_n)$ be given by  \[\mathbb{P}(X_n=i, Y_n=i| X_{n-}=x, Y_{n-}=y) = G(ix,jy),\] where $G$ is as defined in equations (\ref{transition1}) and (\ref{transition2}).
			
Define a process $(Z_n)_{n \in \mathbb{Z}}$ taking values in $\mathbb{N} \cup \lbrace 0 \rbrace$ by \[Z_n=\pi(X_{n-},Y_{n-}).\]
To estimate $\int \delta(x,y) \,d\mu$ it is sufficient to estimate $\int \chi_{[0]} ~d\mu'$ where $\mu'$ is the induced measure on the $\mathbb{Z}$-valued process. For the $\mathbb{Z}$-valued process, the only valid transitions are from $n$ to $0$ or from $n$ to $n+1$.  
	
	  \begin{figure}[h!]
	  \begin{center}
	  \begin{tikzpicture}
	  \tikzset{node style/.style={state, 
	   		fill=gray!20!white,
	   		rectangle}}   
	   \node[state]               (0)   {0};
	   \node[state, right=of 0]   (1)   {1};
	   \node[state, right=of 1]   (2)   {2};
	   \node[draw=none,  right=of 2]   (2-n) {$\cdots$};
	   \node[state, right=of 2-n] (n)   {$n$};
	   \node[draw=none,  right=of n]     {$\cdots$};
	   \draw[>=latex,
	   auto=left,
	   every loop]
	   (0)   edge[loop left] node {}   (0)
	   (0)   edge[left] node {}   (1)
	   (1)  edge[left] node {}  (2)
	   (2)  edge[left] node {}  (2-n)
	   (2-n) edge[left] node {}  (n)
	   (1)   edge[bend left=60] node[above, pos=0.5, sloped] {}  (0)
	   (2)   edge[bend left=80] node[above, sloped] {}  (0) 
	   (n)   edge[bend left=100] node[above, sloped] {}  (0);
	   \end{tikzpicture}
	   \caption{Valid transitions for the $\mathbb{Z}$-valued process} 
	   \label{fig:Z process}
	   \end{center}
	   \end{figure}
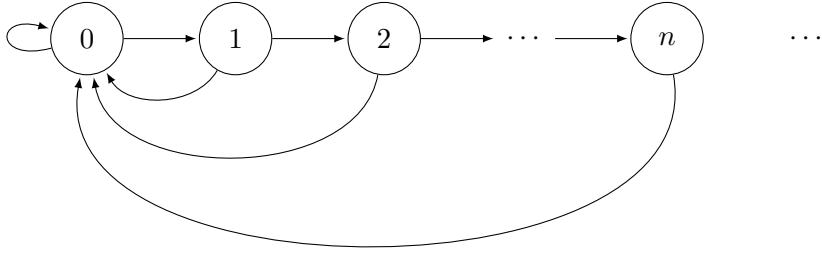

Let $\tau$ be the return time to state $0$. Since the proportion of time spent by the above process in the state 0 is the reciprocal of the expected return time to state zero, we have that \[\int \delta(x,y) \,d\mu= \int \chi_{[0]} ~d\mu' =\frac{1}{\mathbb{E}(\tau)},\]

where, $\mathbb{E}(\tau)$ the expected return time to state $0$. Now 

\[\mathbb{E}(\tau)=\sum_{k=1}^{\infty} k \mathbb{P}(\tau=k)= \sum_{k=1}^{\infty} \mathbb{P}(\tau \geq k).\]

Consider 
\begin{align*}
\mathbb{P} [\tau \geq k] &= \frac{1}{\mu(A_0)} \int_{A_0}\sum_{i_1 \dots i_k} \mathbb{P} \left( X_{n+1} \dots X_{n+k}= Y_{n+1} \dots Y_{n+k}= i_1 \dots i_k | X_{n^-} =x,  Y_{n^-}=y \right) \,d\mu(x,y) \\
&= \frac{1}{\mu(A_0)} \int_{A_0} \sum_{i_1 \dots i_k} F(i_1,x,y) F(i_2,i_1x,i_1y) \dots F(i_k , i_{k-1}\dots i_2i_1x,i_{k-1}\dots i_2i_1y) \,d\mu(x,y).
\end{align*}
Now by (\ref{upper}), we get that \[\sum_{i_j} F(i_j , i_{j-1}\dots i_2i_1x,i_{j-1}\dots i_2i_1y) \geq e^{-(V_{j+1}(g)+d(g,h))}.\]
We get that
\begin{align*}
\mathbb{P} [\tau \geq k ] 
\geq \prod_{j=1}^{k} (e^{-(V_{j+1}(g)+d(g,h))}).
\end{align*}
Hence, \begin{equation} \label{expectation} \mathbb{E}(\tau) \geq \sum_{n=1}^{\infty} \prod_{i=1}^n \left(e^{-(V_{i+1}(g)+d(h,g))}\right) =  \sum_{n=1}^{\infty}\left( \prod_{i=1}^n e^{-V_{i+1}(g)}\right) \left(e^{-d(h,g)}\right)^n.  \end{equation}

This is a power series in $e^{-d(h,g)}$. As $d(h,g) \to 0$, $e^{-d(h,g)} \to 1$. By Abel's theorem,  \[\sum_{n=1}^{\infty} \prod_{i=1}^n \left(e^{-V_{i+1}(g)}\right) (e^{-d(h,g)})^n \to \sum_{n=1}^{\infty} \prod_{i=1}^n e^{-V_{i+1}(g)}.\]
$\sum_{n=1}^{\infty} \prod_{i=1}^n e^{-V_{i+1}(g)} = \infty$ is equivalent to the condition (\ref{eqn:hypothesis}) in the hypothesis, which completes the proof. \qedhere
\end{proof}
We now get Theorem \ref{g-functions and g-measures Lipschitz} as a corollary. 
\begin{proof}[Proof of Theorem \ref{g-functions and g-measures Lipschitz}]
Let $\sum_{i=1}^{\infty} V_{i+1}(g) \leq L_g<\infty$. 
Then by equation (\ref{expectation}), we get that \[\mathbb{E}(\tau) \geq \sum_{n=1}^{\infty} e^{-L_g} (e^{-d(g,h)})^n = e^{-L_g} \frac{e^{-d(g,h)}}{1-e^{-d(g,h)}}= e^{-L_g} (e^{d(g,h)}-1)^{-1}.\]. 
If $d(g,h) \leq \ln(2)$, then $e^{d(g,h)}-1 \leq 2d(g,h)$, we get 
\begin{equation} \label{estimates} \overline{d}(\nu_g,\nu_h)\leq \mathbb{E}(\tau)^{-1}  \leq e^{L_g}(e^{d(g,h)}- 1) \leq 2e^{L_g}d(g,h). \qedhere \end{equation}
\end{proof}
In particular, if $\log(g)$ is $\theta$-weakly Hölder, then $V_k(g) \leq C_g \theta^k$ for all $k \geq 2$, then 
\[\sum_{i=1}^{\infty} V_{i+1}(g) \leq \sum_{i=1}^{\infty} C_g \theta^{i+1} = C_g \frac{\theta^2}{1-\theta}:=L_g.\]
Combining this with equation \ref{estimates}, we get that \[\overline{d}(\nu_g,\nu_h) \leq 2e^{L_g} d(g,h).\] 

\section{Potentials and $\overline{d}$-continuity} \label{potentials}
In this section, we will prove Theorem \ref{Potential and normalized potentials}. Theorem \ref{Potential and equilibrium states} then follows by combining this with Theorem \ref{g-functions and g-measures countable}. 

We prove Theorem \ref{Potential and normalized potentials} by applying the perturbation theory of linear operators to the RPF operator. The proof is based on the spectral gap property (SGP) of the Ruelle operator. For a strongly positive recurrent potential $\phi$, there exists a `rich' Banach space $\mathscr{B}$ of continuous functions on which $L_{\phi} = \lambda P_{\phi} + N_{\phi}$ where $\lambda = P_G (\phi)$ and $P_{\phi}$ is a projection whose image is a one-dimensional subspace, see definition (\ref{spectral property}). We first show that the Ruelle operator depends continuously on the potential in the operator norm coming from the Banach space. This is Lemma  \ref{operator lemma}. The continuity result follows once we show that the projection varies continuously with the potential. This is Lemma \ref{eigenfunction lemma}.

\subsection{Spectral Gap Property} \label{SGP}
We begin by first describing the Spectral Gap Property and the Banach space on which the Ruelle operator associated with a SPR potential acts with a spectral gap.  Recall that the Ruelle operator associated to potential $\phi$ is defined  as \[L_\phi(f)(x)= \sum_{\sigma y=x} e^{\phi(y)} f(y).\]
Let dom$(L_\phi)$ denote the collection of all functions $f$ such that the sum converges for all $x$ in $X$. For the following terminology and results, see \cite{CS09}. For the remainder of the discussion, we will let $\phi$ be a $\theta$-weakly Hölder continuous potential with finite Gurevich pressure.   

\begin{definition} \label{spectral property}
We say that $\phi$ has the spectral gap property (SGP) if there is a Banach space $\mathscr{B}$ of continuous functions on $X$ such that: 
\begin{enumerate}
\item $\mathscr{B}$ is ``rich" 
   \begin{enumerate} 
   \item $\mathscr{B} \subset$ dom$(L_\phi)$ and $\lbrace \mathbbm{1}_{[a]} \colon [a] \neq \emptyset \rbrace \subset \mathscr{B}$
   \item $f \in \mathscr{B} \implies |f| \in \mathscr{B}$ and $\| |f| \|_{\mathscr{B}} \leq \| f \|_{\mathscr{B}}$. 
   \item $\mathscr{B}$-convergence implies convergence on cylinders
   \end{enumerate}
\item  ``Spectral gap":  $L_\phi$ is a bounded operator on $\mathscr{B}$. Furthermore, $L_\phi=\lambda P+N$, where $\lambda= \exp(P_G(\phi)), PN=NP, P^2=P, \dim(\Image P)=1$ and the spectral radius of $N$ is less than $\lambda$.
\item ``Analytic Perturbations": If $g$ is $\theta$-Hölder, then $L_\phi \colon \mathscr{B} \to \mathscr{B}$ is bounded and $z \to L_{\phi+zg}$ is analytic on some complex neighborhood of zero.
\end{enumerate}
\end{definition}
Strong positive recurrence is a necessary and sufficient condition for the spectral gap property. 
\begin{theorem} [\cite{CS09}]
Let $X$ be a topologically mixing countable Markov shift and $\phi$ be a $\theta$-weakly Hölder continuous potential (for some $0<\theta<1$) with finite Gurevich pressure, then $\phi$ has the spectral gap property if and only if $\phi$ is strongly positive recurrent. 
\end{theorem}
Fix $0<\theta<1$. Let $\phi$ be a $\theta$-weakly Hölder continuous potential with finite Gurevich pressure that is strongly positive recurrent. We describe a space $\mathscr{B}$ on which $L_\phi$ acts with a spectral gap.

Fix a state $a$ such that $\Delta_a[\phi]>0$. Without loss of generality, we can assume that $P_G(\phi)=0$. We induce on the state as described in Section \ref{grpf}. Then $P_G(\overline{\phi})=0$. By strong positive recurrence, there exists $\epsilon_a>0$ such that $0<P_G(\overline{\phi+2\epsilon_a})< \infty$. Also since $t \to P_G(\overline{\phi+t})$ is continuous, we can choose $\epsilon_a>0$ small enough such that 
\begin{equation} \label{eq:epsilon} 0<\theta e^p< 1, ~\text{where}~p:=P_G(\overline{\phi+\epsilon_a}). \end{equation}
Let \begin{equation} \label{psy} \psi:= \phi+\epsilon_a-p \mathbbm{1}_{[a]}. \end{equation} Then $\psi$ is strongly positive recurrent and $P_G(\psi)=0$. By Generalized Ruelle Perron Frobenius Theorem, there exists a Borel measure $\nu_0$ which is finite and positive on cylinders, and a positive continuous function $h_0:X \to \mathbb{R}$, such that \[L_\psi^* \nu_0=\nu_0, ~ L_\psi h_0=h_0, ~\text{and}~ \int h_0 \,d\nu_0=1.\] Moreover, $\var_1[\log h_0] \leq \sum_{k \geq 2} \var_k(\phi)$. Setting $c:= \exp(\sum_{k \geq 2} \var_k(\phi))$, we see that for every $x$, \begin{equation} \label{bound} h_0(x)= c^{\pm 1} h_0[x_0],\end{equation} where $h_0[x_0] := \sup_{[x_0]} h_0$
 
Define for $x,y \in X$, 
\begin{align*}
t(x,y) &:= \min \lbrace n \colon x_n \neq y_n \rbrace ~\text{where} \min \emptyset= \infty, ~\text{and}\\
s_a(x,y) &:= \# \lbrace 0 \leq i \leq t(x,y)-1 \colon x_i=y_i=a \rbrace. 
\end{align*}

Let $\mathscr{B}$ be the collection of continuous functions $f:X \to \mathbb{C}$ for which 
\[\|f\|_{\mathscr{B}}:= \sup_{b \in S} \frac{1}{h_0[b]} \left[ \sup_{x \in [b]} |f(x)| +\sup_{x \neq y \in [b]} \left\lbrace \frac{|f(x)-f(y)|}{\theta^{s_a(x,y)}} \right\rbrace\right] < \infty.\]

$L_\phi$ acts with a spectral gap on $(\mathscr{B}, \| \cdot \|_{\mathscr{B}})$.

\subsection{Operator depends continuously on potential} 
Let $\| \cdot \|_{\theta}$ denote the Hölder norm given by $\| \phi \|_{\theta}:= \| \phi \|_{\infty}+ |\phi|_{\theta}$, where $\| \cdot \|_{\infty}$ is the uniform norm and $|\phi |_{\theta}:= \sup \lbrace |\phi(x)- \phi(y)| / \theta^{t(x,y)} \colon x,y \in X, x \neq y\rbrace$.

Let $\phi$ and $\tau$ be locally $\theta$-Hölder-continuous strongly positive recurrent potentials with finite Gurevich pressure. 
Fix a state $a$ such that $\Delta_a[\phi]>0$ and $\Delta_a[\tau]>0$ . Without loss of generality, assume that $P_G(\phi)=P_G(\tau)=0$ (otherwise we pass to $\phi-P_G(\phi)$ and $\tau-P_G(\tau)$). 
Consider 
\[\|L_{\phi}-L_{\tau} \| = \sup_{\|f\|_{\mathscr{B}}=1} \|L_{\phi}f - L_{\tau}f \|_{\mathscr{B}}. \]

\begin{lemma} \label{operator lemma}
Given $\phi$ and $\beta$ as above, for given $\epsilon >0$, there exists $\delta>0$ such that $\|L_{\phi}-L_{\tau} \|< \epsilon$, whenever $\tau \in A_{\theta}$ is such that $\|\phi-\tau\|_{\theta}<\delta$.
\end{lemma}
\begin{proof}[Proof of Lemma \ref{operator lemma}]  
Let $f \in \mathscr{B}$ with $\| f \|_{\mathscr{B}}=1$. Now, 
\begin{multline} \label{norm} 
\|L_{\phi}f - L_{\tau}f \|_{\mathscr{B}} = \sup_{b \in S}  \frac{1}{h_0[b]} ~~ \sup_{x \in [b]} |L_{\phi}f(x)-L_{\tau}f(x)|  \\ + \sup_{x,y \in [b], x \neq y} \lbrace |L_{\phi}f(x)-L_{\tau}f(x)-L_{\phi}f(y)+L_{\tau}f(y)|/\theta^{s_a(x,y)} \rbrace. \end{multline}
Since $h_0$ is an eigenfunction for $L_{\psi}$, we have that $h_0(x)=\sum_{\sigma(y)=x} e^{\psi(y)}h_0(y)$, where $\psi$ is as given by equation (\ref{psy}). Furthermore, by the definition of the norm $\| \cdot\|_{\mathscr{L}}$ and (\ref{bound})
\begin{equation} \label{eqn:norm} |f(y)| \leq \|f\|_{\mathscr{B}} h_0[b] \leq c|h_0(y)|. \end{equation}

By definition $\psi=\phi+\epsilon_a-p\mathbbm{1}_{[a]}$, so 
\begin{equation} \label{eqn:phi} \phi \leq \psi+p. \end{equation}

For $x \in [b]$, we have the following calculation, 
\begin{align} \label{calc}
\left| \sum_{r \to b} e^{\phi(rx)}f(rx) \right|  &\leq  \sum_{r \to b} e^{\phi(rx)} |f(rx)| \nonumber \\
&\leq  \sum_{r \to b} e^{\phi(rx)}ch_0(rx) ~(\text{using}~ \ref{eqn:norm}) \nonumber \\
& \leq c e^p  \sum_{r \to b} e^{\psi(rx)}h_0(rx) ~(\text{using}~ \ref{eqn:phi}) \nonumber \\
\left| \sum_{r \to b} e^{\phi(rx)}f(rx) \right| &\leq  c e^p h_0(x) ~(\text{using}~ L_{\psi} h_0 = h_0).
\end{align}
We begin by considering the first term in the (\ref{norm}). Let $g=\tau-\phi$. Fix a state $b \in S$ and consider for $x \in [b]$
 \begin{align}
\left| \frac{L_{\phi}f(x)-L_{\tau}f(x)}{h_0(x)} \right|&= \left| \frac{\sum_{\sigma(y)=x}\left(e^{\phi(y)}-e^{\tau(y)}\right)f(y)}{h_0(x)}  \right| \nonumber \\
 &= \left|  \frac{\sum_{\sigma(y)=x}e^{\phi(y)}(1-e^{g(y)})f(y)}{h_0(x)}  \right| \nonumber \\
&\leq \frac{\sum_{\sigma(y)=x} e^{\phi(y)}\left|1-e^{g(y)}\right| |f(y)| }{h_0(x)} \nonumber  \\
&\leq \left(e^{\|\phi-\tau \|_{\infty}}-1 \right)   \frac{\sum_{\sigma(y)=x} e^{\phi(y)} |f(y)| }{h_0(x)} \nonumber  \\
&\leq c \left(e^{\|\phi-\tau \|_{\infty}}-1 \right)  e^p  ~(\text{using}~ \ref{calc}) \nonumber \\ 
&\leq c (C\| \phi - \tau \|_{\infty}) e^p ~(\text{for} ~ \|\phi-\tau\|_{\infty} <1).
\end{align}
Since $h_0[b] \geq c^{-1} h_0(x)$ for all $x \in [b]$, we get the following upper bound for the first term in equation (\ref{norm}).
\begin{equation} \label{first term}
\sup_{b \in S}  \frac{1}{h_0[b]} ~~ \sup_{x \in [b]} |L_{\phi}f(x)-L_{\tau}f(x)| \leq C \| \phi - \tau \|_{\infty} e^p.
\end{equation}
Now, we consider the second term in the (\ref{norm}). 
For $x,y $ in $[b]$ with $x \neq y$, consider  
\begin{align}
&|L_{\phi}f(x)-L_{\tau}f(x)-L_{\phi}f(y)+L_{\tau}f(y)|/\left(h_0[b]\theta^{s_a(x,y)}\right) \nonumber \\
&= \left|\frac{\sum_{r \to b} e^{\phi(rx)}f(rx)-e^{\tau(rx)}f(rx)-e^{\phi(ry)}f(ry)+e^{\tau(ry)}f(ry)}{\theta^{s_a(x,y)} h_0[b]}. \right| 
\end{align}
We add and subtract $e^{\phi(rx)}f(rx)e^{(\tau-\phi)(ry)}$ from the quantity inside the sum in the numerator, and rewrite it as follows: 
\begin{align*}
&\left|e^{\phi(rx)}f(rx)-e^{\tau(rx)}f(rx)-e^{\phi(ry)}f(ry)+e^{\tau(ry)}f(ry)\right|  \\
&= \left|e^{\phi(rx)}f(rx)(e^{(\tau-\phi)(ry)}-e^{(\tau-\phi)(rx)})+e^{\phi(rx)}f(rx)(1-e^{(\tau-\phi)(ry)})-e^{\phi(ry)}f(ry)(1-e^{(\tau-\phi)(ry)}) \right| \\
&= | e^{\phi(rx)}f(rx)(e^{(\tau-\phi)(ry)}-e^{(\tau-\phi)(rx)})-(1-e^{(\tau-\phi)(ry)})(e^{\phi(ry)}f(ry)-e^{\phi(rx)}f(rx))| 
\end{align*}
Using this along with $s_a(x,y) \leq t(x,y)$, we get that
\begin{equation} \label{estimate}
|L_{\phi}f(x)-L_{\tau}f(x)-L_{\phi}f(y)+L_{\tau}f(y)|/\left(h_0[b]\theta^{s_a(x,y)}\right)  \leq A+B
\end{equation}
where 
\begin{align*}
A &= \left|\frac{\sum_{r \to b} e^{\phi(rx)}f(rx)(e^{(\tau-\phi)(ry)}-e^{(\tau-\phi)(rx)})}{\theta^{t(x,y)} h_0[b]}\right| \\
B &= \left|\frac{\sum_{r \to b}(1-e^{(\tau-\phi)(ry)})(e^{\phi(ry)}f(ry)-e^{\phi(rx)}f(rx))}{\theta^{t(x,y)} h_0[b]}\right|.
\end{align*}
We first estimate the first term in the above expression (\ref{estimate}). Note that for $\| \tau-\phi\|_{\infty} \leq 1$, there exists $C_0$ such that 
\begin{align}  \label{holder norm}
\left| e^{(\tau-\phi)(ry)}-e^{(\tau-\phi)(rx)} \right| &\leq C_0 \left| (\tau-\phi)(ry)-(\tau-\phi)(rx) \right| \nonumber \\
&\leq C_0 \|\tau-\phi \|_{\theta} \theta^{t(x,y)} .
\end{align}
Consider, 
\begin{align*}
A & \leq C_0 \|\tau-\phi\|_{\theta} \theta^{t(x,y)}  \frac{\sum_{r \to b} e^{\phi(rx)}|f(rx)|}{\theta^{t(x,y)} h_0[b]} ~~(\text{using}~ (\ref{holder norm})) \\
& \leq C_1 \|\tau-\phi\|_{\theta} e^p ~~\left(h_0[b] = \sup_{[b]} h_0 ~ \text{and using} ~ (\ref{calc}) \right).
\end{align*} 
Next let us consider the second term in equation (\ref{estimate}). We have that 
\begin{equation} \label{a}
B \leq C_2 \|\tau-\phi\|_{\infty} \frac{ \sum_{r \to b} \left|(e^{\phi(ry)}f(ry)-e^{\phi(rx)}f(rx))\right|}{\theta^{t(x,y)} h_0[b]}.
\end{equation}
Now
\begin{align} \label{b}
\left|e^{\phi(ry)}f(ry)-e^{\phi(rx)}f(rx)\right| &= \left|e^{\phi(ry)}f(ry)-e^{\phi(rx)}f(ry)+e^{\phi(rx)}f(ry)-e^{\phi(rx)}f(rx) \right| \nonumber \\
&= \left| (e^{\phi(ry)}-e^{\phi(rx)})f(ry)+e^{\phi(rx)}(f(ry)-f(rx)) \right|.
\end{align}
Since 
\begin{equation} \label{eqn:estimates} \left| f(rx)- f(ry) \right| \leq \| f\|_{\mathscr{B}} h_0[r] \theta^{t(x,y)} \leq h_0[r] \theta^{t(x,y)}, \end{equation}
we get
\begin{align} \label{c}
\sum_{r \to b} \left|e^{\phi(rx)}(f(ry)-f(rx)) \right| &\leq \sum_{r \to b} \left| e^{\phi(rx)} h_0[r] \theta^{t(x,y)} \right| (\text{using}~ (\ref{eqn:estimates})) \nonumber\\
&\leq c \sum_{r \to b} e^{\phi(rx)} h_0(rx) \theta^{t(x,y)} ~(\text{using}~ (\ref{bound})) \nonumber \\
&\leq c \theta^{t(x,y)} e^p \sum_{r \to b} e^{\psi(rx)} h_0(rx) ~(\text{using}~ (\ref{eqn:phi})) \nonumber \\
&\leq c \theta^{t(x,y)} e^p h_0(x) ~~(h_0 = L_{\psi} h_0) \nonumber \\
&\leq c \theta^{t(x,y)} e^p h_0[b] ~~\left(h_0[b] = \sup_{[b]} h_0 \right)
\end{align}
Since $\phi$ is locally Hölder continuous, there exists $C_{\phi}$ such that for all $x,y$ and $r$, we have $|\phi(rx)-\phi(ry)| \leq C_{\phi}\theta^{t(x,y)}$. So, 
\begin{align} \label{d}
 \left| \sum_{r \to b} (e^{\phi(ry)}-e^{\phi(rx)})f(ry) \right| & \leq \sum_{r \to b} e^{\phi(ry)} |(1-e^{\phi(rx)-\phi(ry)})| |f(ry)|  \nonumber \\
&\leq c e^p \sum_{r \to b} e^{\psi(ry)} |(1-e^{\phi(rx)-\phi(ry)})|  h_0(ry)  ~(\text{similar to}~ \ref{calc}) \nonumber  \\
&\leq C_3 e^p \sum_{r \to b} e^{\psi(ry)} \left| \phi(rx)-\phi(ry) \right|  h_0(ry) \nonumber  \\
&\leq C_3 e^p C_{\phi} \theta^{t(x,y)} h_0(y) \nonumber ~(L_{\psi}h_0 = h_0) \nonumber \\
&\leq C_3 e^p C_{\phi} \theta^{t(x,y)} h_0[b]  \left(h_0[b] = \sup_{[b]} h_0 \right).
\end{align}
Combining equations (\ref{a}), (\ref{b}), (\ref{c}) and (\ref{d}), we get the following estimate on the second term in  (\ref{estimate})
\begin{equation} \label{II}
\left|\frac{\sum_{r \to b}(1-e^{(\tau-\phi)(ry)})(e^{\phi(ry)}f(ry)-e^{\phi(rx)}f(rx))}{\theta^{t(x,y)} h_0[b]} \right| \leq C_2 \|\tau-\phi\|_{\infty} \left[ c e^p + C_3 C_{\phi} \right].
\end{equation}

Combining equations (\ref{norm}), (\ref{first term}) and (\ref{II}), we get
\[\|L_{\phi}f - L_{\tau}f \|_{\mathscr{B}}  \leq C \| \phi - \tau \|_{\infty} e^p +C_1 \|\tau-\phi\|_{\theta} e^p+ C_2 \|\tau-\phi\|_{\infty} \left[ c e^p + C_3 C_{\phi} \right]. \]
This completes the proof. 
\end{proof}

\subsection{Eigenfunctions depend continuously on potential} 
Let $\phi$ and $\tau$ be locally $\theta$-Hölder-continuous strongly positive recurrent potentials with finite Gurevich pressure. 
We extend the result of this section to weakly Hölder potentials at the beginning of section \ref{completion}. Let $h_{\phi}, h_{\tau}$ be positive continuous functions given by the Generalized RPF Theorem. Then $L_{\phi}h_{\phi}= h_{\phi}$ and $L_{\tau}h_{\tau}=h_{\tau}$.  

\begin{lemma} \label{eigenfunction lemma}
Let $\phi$ be a locally Hölder, SPR potential with finite Gurevich pressure. Given $\alpha>0$, there exists $\delta>0$, such that for $\tau \in A_{\theta}$ such that $d_{\theta}(\phi,\tau) < \delta$, we get that \begin{equation} \label{eigenfunction conclusion} \left\| \frac{h_{\phi}}{h_{\phi} \circ \sigma} \frac{h_{\tau} \circ \sigma}{h_{\tau}} - 1 \right\|_{\infty} < \alpha. \end{equation} \end{lemma}
Note that (\ref{eigenfunction conclusion}) then holds for all scalar multiples of $h_{\phi}$ and $h_{\tau}$. 
\begin{proof}[Proof of Lemma \ref{eigenfunction lemma}]
Let $\ms{B}$ be the Banach space on which $L_{\phi}$ acts with a spectral gap. According to the theory of analytic perturbations of linear operators, there exists an $\epsilon>0$ such that every $L: \mathscr{B} \to \mathscr{B}$ which satisfies $\| L - L_{\phi} \| < \epsilon$ has a decomposition of the form $L=\lambda(L) P(L)+N(L)$, where $P(L)$ and $N(L)$  are bounded operators such that $P(L)^2 =P(L), P(L)N(L)=N(L)P(L), \dim(Im P(L))=1$ and the spectral radius of $N(L)$ is less than $\lambda(L)$. Moreover, if $\epsilon > 0$ is small enough, then $L \to P(L), N(L), \lambda(L)$ are analytic on $\lbrace L \colon \|L-L_{\phi} \|< \epsilon \rbrace$. 

By lemma (\ref{operator lemma}), there exists $\delta >0$, such that for $\|\phi-\rho\|_{\theta} < \delta$, we get that $\|L_{\phi}-L_{\rho}\| < \epsilon$, so $L_{\rho}=\lambda_{\rho} P_{\rho}+ N_{\rho}$, where $\rho \to P_{\rho}$ is continuous. 

Let $\gamma>0$ be given, then there exists $\delta>0$ such that for $\|\rho-\phi\|_{\theta}< \delta$, $\| P_{\rho}-P_{\phi}\|<\frac{\gamma}{2}$. Let $\tau$ be such that $\|\tau-\phi\|_{\theta}< \frac{\delta}{2}$.

By equation (\ref{eq:epsilon}), $p \to 0$ as $\epsilon_a \to 0$. Further, by equation (\ref{psy}), $\phi-p \leq \psi \leq \phi+\epsilon_a$, so if take $\epsilon_a, p \in (0, \frac{\delta}{4})$, we get that $\| \psi - \phi \|_{\infty}< \frac{\delta}{2}$. Moreover, $\psi-\phi = \epsilon_a-p \mathbbm{1}_{[a]}$, so for all $n \geq 1, \var_n(\psi-\phi)=0$. So, $\| \psi - \phi \|_{\theta}< \frac{\delta}{2}$. We get that $\| P_{\psi}-P_{\phi}\|<\frac{\gamma}{2}$. Further, $\| \psi -\tau\|_{\theta} <\delta$ and $\| P_{\psi}- P_{\tau} \| < \gamma$.

Normalize so that the corresponding eigenfunctions $h_{\tau} $ and $h_{\phi}$ satisfy $P_{\psi} h_{\phi}= P_{\psi} h_{\tau}= h_0$. Then 
\begin{align*}
\|h_{\phi}  \|_{\mathscr{B}} - \|h_0 \|_{\mathscr{B}} \leq \| h_{\phi} - h_0 \|_{\mathscr{B}} &= \| P_{\phi} h_{\phi}-P_{\psi} h_{\phi}  \|_{\mathscr{B}} \\
&\leq \| P_{\phi} - P_{\psi} \| \| h_{\phi} \|_{\mathscr{B}} < \gamma \| h_{\phi} \|_{\mathscr{B}} \\
\implies \|h_{\phi}  \|_{\mathscr{B}} &\leq \frac{1}{1-\gamma} \|h_0  \|_{\mathscr{B}}.
\end{align*}

Similarly \[ \|h_{\tau}  \|_{\mathscr{B}} \leq \frac{1}{1-\gamma} \|h_0  \|_{\mathscr{B}}.\]

We want to show that $\left\| \frac{h_{\phi}}{h_{\phi} \circ \sigma} \frac{h_{\tau} \circ \sigma}{h_{\tau}}-1 \right\|_{\infty}$ is small.
For $y \in \Sigma$ and for $\gamma \leq \frac{1}{2}$, we have that

\begin{align*}
\left| \left(\frac{h_\phi}{h_0} -1 \right)(y) \right| &= \left| \frac{h_{\phi}(y)-h_{0}(y)}{h_0(y)} \right| \\
&\leq c \|h_{\phi}-h_0 \|_{\mathscr{B}} \\
&\leq c \gamma \|h_{\phi}\|_{\mathscr{B}} \\
&\leq c \frac{\gamma}{1-\gamma} \|h_0\|_{\mathscr{B}} \\
&\leq 2\gamma c \|h_0\|_{\mathscr{B}}.
\end{align*}

So, \[h_0(y)\left(1-2c\gamma \|h_0\|_{\mathscr{B}} \right) \leq h_{\phi}(y) \leq h_0(y)\left(1+2c \gamma \|h_0\|_{\mathscr{B}} \right). \]
Similarly, \[h_0(y)\left(1-2c\gamma \|h_0\|_{\mathscr{B}} \right) \leq h_{\tau}(y) \leq h_0(y)\left(1+2c \gamma \|h_0\|_{\mathscr{B}} \right). \]

Given $1>\alpha>0$, choose $\beta_1>0$ s.t. $\left(\frac{1+\beta_1}{1-\beta_1}\right)^2=1+\alpha$ and $\beta_2>0$ such that $1-\alpha=\left(\frac{1-\beta_2}{1+\beta_2}\right)^2$. Now choose $\gamma \in \left(0,\frac{1}{2} \right)$ such that $2c\gamma \|h_0\|_{\mathscr{B}}< \min(\beta_1,\beta_2)$. 

Consider 
\begin{align*}
\frac{h_{\phi}(y)}{h_{\tau}(y)}  &\leq \frac{\left(1+2c \gamma \|h_{0} \|_{\mathscr{B}} \right)}{\left(1-2c \gamma \|h_{0} \|_{\mathscr{B}} \right)}  \\
 & \leq \frac{1+\beta_1}{1-\beta_1}.
\end{align*}
Similarly, $\frac{h_{\tau}(y)}{h_{\phi}(y)} \leq \frac{1+\beta_1}{1-\beta_1}$.

Hence, we get that \[\frac{h_{\phi}(y)}{h_{\phi} \circ \sigma (y)} \frac{h_{\tau} \circ \sigma(y)}{h_{\tau(y)}} \leq  \left(\frac{1+\beta_1}{1-\beta_1}\right)^2 = 1+\alpha.\]

Also, \begin{align*}
\frac{h_{\phi}(y)}{h_{\tau}(y)} &\geq \frac{\left(1-2c \gamma \|h_{0} \|_{\mathscr{B}} \right)}{\left(1+2c \gamma \|h_{0} \|_{\mathscr{B}} \right)}  \\
 & \geq \frac{1-\beta_2}{1+\beta_2}.
\end{align*}
Similarly, $\frac{h_{\tau}(y)}{h_{\phi}(y)} \geq \frac{1-\beta_2}{1+\beta_2}$.

So, we get that \[\frac{h_{\phi}(y)}{h_{\phi} \circ \sigma (y)} \frac{h_{\tau} \circ \sigma(y)}{h_{\tau(y)}} \geq  \left(\frac{1-\beta_2}{1+\beta_2}\right)^2 = 1-\alpha.\] 

Hence, for all $y$ in $\Sigma$, \[ \left| \frac{h_{\phi}(y)}{h_{\phi} \circ \sigma (y)} \frac{h_{\tau} \circ \sigma(y)}{h_{\tau(y)}} - 1 \right| \leq \alpha. \]

Therefore, \[ \left\| \frac{h_{\phi}}{h_{\phi} \circ \sigma} \frac{h_{\tau} \circ \sigma}{h_{\tau}} - 1 \right\|_{\infty} < \alpha. \qedhere \] 
\end{proof}

\subsection{Completion of proof} \label{completion}
We begin by extending Lemma \ref{eigenfunction lemma} to weakly Hölder potentials. Let $\tilde{S}=S \times S$ and define $\tilde{\pi}: S^{\mathbb{N} \cup \lbrace 0 \rbrace} \to \tilde{S}^{\mathbb{N} \cup \lbrace 0 \rbrace}$ by \[\tilde{\pi}(x_0 x_1 x_2 \dots)= (x_0,x_1)(x_1,x_2) \dots \]
Let $\tilde{\Sigma}= \tilde{\pi}(\Sigma)$. Then $\tilde{\pi}$ is a topological conjugacy from $(\Sigma, \sigma)$ to $(\tilde{\Sigma}, \sigma)$. Given a potential $\phi : \Sigma \to \mathbb{R}$, let $\tilde{\phi}=\phi \circ \tilde{\pi}: \tilde{\Sigma} \to \mathbb{R}$. 
If $\tilde{x}, \tilde{y} \in \tilde{\Sigma}$ are such that $\tilde{x_1} \dots \tilde{x_k}=  \tilde{y_1} \dots \tilde{y_k}$, $x=\tilde{\pi}(\tilde{x})$ and $y=\tilde{\pi}(\tilde{y})$, we get that $x_1 \dots x_{k+1}= y_1 \dots y_{k+1}$. So, 
\[|\tilde{\phi}(\tilde{x})-\tilde{\phi}(\tilde{y})| = |\phi(x)-\phi(y)| \leq \var_{k+1}(\phi).\]
Therefore, $\var_{k}(\tilde{\phi}) \leq \var_{k+1}(\phi)$. Thus, if $\phi$ is weakly Hölder continuous, then $\tilde{\phi}$ is locally Hölder continuous. Further, if $L_{\phi}h_{\phi}= \lambda h_{\phi}$, we get that 
\begin{align*}
L_{\tilde{\phi}} \tilde{h}(\tilde{x}) &= \sum_{r \to x_0} e^{\tilde{\phi}((r,x_0)\tilde{x})} \tilde{h}((r,x_0)\tilde{x}) \\
&= \sum_{r \to x_0} e^{\phi(rx)} h(rx) =L_{\phi} h(x) \\
&= \lambda h(x) = \lambda \tilde{h}(x). 
\end{align*}
So, the conclusion of Lemma \ref{eigenfunction lemma} holds for $h_{\tilde{\phi}}$ and $h_{\tilde{\tau}}$, if it holds for $h_{\phi}$ and $h_{\tau}$. Hence, Lemma \ref{eigenfunction lemma} holds for weakly Hölder potentials.

Finally, to complete the proof of Theorem \ref{Potential and normalized potentials}, consider the $g$-functions given by $g_{\phi}= \frac{e^{\phi}h_{\phi}}{ h_{\phi} \circ \sigma}$ and $g_{\tau}= \frac{e^{\tau}h_{\tau}}{ h_{\tau} \circ \sigma}$. We have
\[\left\| \log(g_{\phi})- \log(g_{\tau})\right\|_{\infty} \leq \|\phi-\tau\|_{\infty} + \left\|  \log \left( \frac{h_{\phi}}{h_{\phi} \circ \sigma} \frac{h_{\tau} \circ \sigma}{h_{\tau}} \right)  \right\|_{\infty}. \]

It follows from Lemma (\ref{eigenfunction lemma}) that $\left\| \log(g_{\phi})- \log(g_{\tau})\right\|_{\infty} \to 0$ as $d_{\theta}(\phi,\tau) \to 0$.

Since $\log(g_{\phi})$ and $\log(g_{\tau})$ have summable variations, so $g_{\phi}$ and $g_{\tau}$ satisfy the condition of Theorem \ref{g-functions and g-measures countable}. The RPF measures $m_{\phi}$ and $m_{\tau}$ are \emph{g}-measures corresponding to $g_{\phi}$ and $g_{\tau}$. This completes the proof of Theorem \ref{Potential and equilibrium states}.

\newpage

\bibliographystyle{amsalpha}
\bibliography{references1}

\end{document}